\documentclass[11pt]{amsart}

\usepackage{amsmath,amssymb,amsthm,tikz}
\usetikzlibrary{calc,decorations.markings,matrix,arrows}

\newcommand{\C}{\mathbb{C}}

\newcommand{\cK}{\mathcal{K}}
\newcommand{\cO}{\mathcal{O}}

\newcommand{\Gr}{\mathrm{Gr}}
\newcommand{\Hom}{\mathrm{Hom}}

\newcommand{\Irr}{\mathrm{Irr}}
\newcommand{\SL}{\mathrm{SL}}

\renewcommand{\top}{\mathrm{top}}
\newcommand{\dom}{\mathrm{dom}}

\newcommand{\vG}{G^\vee}
\newcommand{\vlambda}{{\vec{\lambda}}}
\newcommand{\vgamma}{{\vec{\gamma}}}
\newcommand{\vmu}{{\vec{\mu}}}
\newcommand{\vnu}{{\vec{\nu}}}

\newcommand{\braket}[1]{{\langle #1 \rangle}}

\newcommand{\eq}[2]{\begin{equation}\label{#1}#2\end{equation}}

\newcommand{\bv}{\mathbf{v}}
\newcommand{\bw}{\mathbf{w}}
\newcommand{\cM}{\mathcal{M}}
\newcommand{\fg}{\mathfrak{g}}

\newtheorem{theorem}{Theorem}[section]
\newtheorem{proposition}[theorem]{Proposition}
\newtheorem{lemma}[theorem]{Lemma}

\newtheorem{corollary}[theorem]{Corollary}

\tikzset{midto/.style={postaction={decorate,
    decoration={markings,mark=at position .5 with
    {\draw (-.035,-.07) -- (.035,0) -- (-.035,.07);}}}}}
\tikzset{midfrom/.style={postaction={decorate,
    decoration={markings,mark=at position .5 with
    {\draw (.035,-.07) -- (-.035,0) -- (.035,.07);}}}}}
\tikzset{web/.style={darkblue,semithick}}

\begin{document}

\title{Cyclic sieving, rotation, and geometric representation theory}
\author{Bruce Fontaine}
\address{Mathematical Sciences Research Institute}
\email{bfontain@gmail.com}
\author{Joel Kamnitzer}
\address{University of Toronto}
\email{jkamnitz@math.toronto.edu}

\begin{abstract}
We study rotation of invariant vectors in tensor products of minuscule representations.  We define a combinatorial notion of rotation of minuscule Littelmann paths.  Using affine Grassmannians, we show that this rotation action is realized geometrically as rotation of components of the Satake fibre.  As a consequence, we have a basis for invariant spaces which is permuted by rotation (up to global sign).  Finally, we diagonalize the rotation operator by showing that its eigenspaces are given by intersection homology of quiver varieties.  As a consequence, we generalize Rhoades' work on the cyclic sieving phenomenon.
\end{abstract}
\maketitle

\section{Introduction}
\subsection{Cyclic Sieving Phenomenon}
Let $ X $ be a finite set and let $ c \in S_X $ be a permutation of $ X $ of order $ r $.  We can study the sizes of the fixed point set $ X^{c^d}$ for $ d \ge 0 $.  Let $ f(q) $ be a polynomial.  Following Reiner-Stanton-White, we say that $ (X, c, f(q)) $ exhibits the \textit{cyclic sieving phenomenon} if $ f(\zeta^d) = | X^{c^d} | $ for all $ d \ge 0 $, where $ \zeta $ is a fixed primitive $ r$-th root of unity.

In the main theorem of \cite{R}, Rhoades proved that the set of semistandard Young tableaux of fixed rectangular shape and fixed content exhibits the cyclic sieving phenomenon, where the permutation is given by some power of the promotion bijection.  He proved the cyclic sieving polynomial is the Kostka-Foulkes polynomial (up to a power of $ q$). The goal of the present paper is to generalize Rhoades' result, simplify his proof, and situate the result in the context of geometric representation theory.

Though the definition of the cyclic sieving phenomenon only involves finite sets, in order to find a nice expression for $ f(q) $ it is convenient to replace $ X $ by the vector space $ \C X $ spanned by $ X $.  The permutation $ c $ extends to a diagonalizable linear operator $ c : \C X \rightarrow \C X $.  Note that the size of the fixed point sets are given by the traces of the powers of $ c^d $ acting on $ \C X $.  Thus if we can diagonalize $ c $ on $ \C X $, then we can find the cyclic sieving polynomial.

When $ X $ is the above set of tableaux, the space $\C X $ can be identified with the space of invariants in a tensor product of representations of $ SL_n $.  Under this identification $ c $ becomes, up to sign, the linear operator of rotation of tensor factors.  Thus we are lead to studying the linear operator of rotation on the vector space of invariant vectors in a tensor product.

\subsection{Minuscule Littelmann paths and their rotations}
Fix a complex semisimple group $ G $.  Recall that a dominant weight $ \lambda $ is called minuscule if all weights of $ V(\lambda) $ lie in the Weyl group orbit of $ \lambda $.  A \textit{minuscule sequence} is a sequence $ \vlambda = (\lambda_1, \dots, \lambda_m) $ of minuscule weights.  The \textit{rotation} $ \vlambda^{(j)} $ of a minuscule sequence is defined be the effect of rotating $ \vlambda $ by $ j $ steps, so that $ \vlambda^{(j)}_k = \vlambda_{j + k} $ (where we regard the indices as lying in $ \mathbb{Z}/m $).

In the paper, we will be interested in the tensor product representation
$$ V(\vlambda) := V(\lambda_1) \otimes \cdots \otimes V(\lambda_m)$$
and in particular its space of invariants $ V(\vlambda)^G $.  We will assume that $ |\vlambda | = \sum_k \lambda_k $ lies in the root lattice, so that $ V(\vlambda)^G $ will be non-empty.

A \textit{minuscule Littelmann path} of type $ \vlambda $ is a sequence $$ \vmu = (\mu_0 = 0, \mu_1, \dots, \mu_m= 0) $$ of dominant weights of $ G $ such that $ \mu_k - \mu_{k-1} = \lambda_k $ for all $ k $.  We let $ P(\vlambda) $ denote the set of minuscule Littelmann paths of type $ \vlambda $.  It is an easy observation that the size of $ P(\vlambda) $ equals the dimension of $ V(\vlambda)^G $.

The rotation of tensor factors gives a linear map $ R : V(\vlambda)^G \rightarrow V(\vlambda^{(1)})^G $.  In this paper, we define a combinatorial version of this map: a bijection $ R: P(\vlambda) \rightarrow P(\vlambda^{(1)}) $ we call \textit{rotation}.  When $ G = SL_n $, minuscule Littelmann paths of type $ \vlambda $ are naturally in bijection with semistandard Young tableaux of fixed rectangular shape and content given by $ \vlambda$.  We prove that under this correspondence, rotation of minuscule Littelmann paths corresponds to promotion of tableaux.

Let $ \ell $ be a positive integer such that $ \vlambda^{(\ell)} = \vlambda $.  We let $ r = m/\ell $ and fix $\zeta$, a primitive $ r$-th root of unity.  We will study cyclic sieving for the action of $ R^\ell $ on $ P(\vlambda)$.

\subsection{Affine Grassmannian perspective}
The vector space of invariants $ V(\vlambda)^G $ has two incarnations in geometric representation theory, both of which will be important in this paper.

The geometric Satake correspondence of Lusztig \cite{Lusztig:qanalog}, Ginzburg \cite{Ginzburg:loop}, and Mirkovi\'c-Vilonen \cite{MV:geometric} gives an interpretation of the representation theory of $ G $ in terms of the geometry of $ \Gr $, the affine Grassmannian of the Langlands dual group $ \vG $.

Starting with our minuscule sequence $ \vlambda$, we can consider the Satake fibre $ F(\vlambda) $, which is the variety of $ m$-gons with side lengths $ \vlambda $ in $ \Gr $.  By the geometric Satake correspondence, we obtain a canonical identification
\begin{equation} \label{eq:geomsat}
H_{\top}(F(\vlambda)) = V(\vlambda)^G
\end{equation}
The left hand side has a canonical basis given by the classes of irreducible components.  The resulting basis of $ V(\vlambda)^G $ is called the Satake basis.  In \cite{us}, we constructed a bijection $ \vmu \mapsto Z_\vmu $ between $ P(\vlambda)$ and the irreducible components of $ F(\vlambda)$.  Our first main result examines the relationship between rotation of $ V(\vlambda)^G $ and rotation of minuscule Littelmann paths.

\begin{theorem} \label{th:intro1}
Up to a global sign, rotation of tensor factors takes the Satake basis of $V(\vlambda)^G $ to the Satake basis of $ V(\vlambda^{(1)})^G $, according to the rotation $ R $ of minuscule Littelmann paths.
\end{theorem}

A similar result for the special case $ G = SL_n$ is proved by Rhoades \cite{R} with the Satake basis replaced by the dual canonical basis.  On the other hand, when $ G = SL_2,SL_3 $ a similar result is proved by Petersen-Pylyavskyy-Rhoades \cite{PPR} with the Satake basis replaced by the web basis.

\subsection{Quiver variety perspective}
In this section, we specialize to the case where $ G $ is simply-laced.  By the work of Nakajima, the geometry of quiver varieties can also be used to study the representation theory of $ G $.  In particular, we will study a certain singular affine quiver variety $ \cM_0(\bv, \bw) $, where $ \bw, \bv $ are computed using $\vlambda $.

From Nakajima \cite{Nak99}, we see that there is an isomorphism between the stalk of the IC sheaf of $ \cM_0(\bv, \bw) $ at $ 0 $ and our invariant space $ V(\vlambda)^G $.  Our second main result refines this idea.

\begin{theorem} \label{th:intro2}
There exists an isomorphism $$ \bigoplus_k H^{2k}(IC_0(\cM_0(\bv, \bw))) \cong V(\vlambda)^G $$ which intertwines multiplication by $ \zeta^k $ on $ H^{2k}(IC_0(\cM_0(\bv, \bw))) $ with the rotation $ R^\ell $.
\end{theorem}

Let us define the Kazhdan-Lusztig type polynomial of $ \cM_0(\bv, \bw) $ as $$ f_\vlambda(q) = \sum_k \dim H^{2k}IC_0(\cM_0(\bv, \bw)) q^k. $$  From Nakajima's work, we can consider $ f_\vlambda(q) $ as a ``$q$-analog of tensor product multiplicity''.

\subsection{Conclusion}
Combining Theorems \ref{th:intro1} and \ref{th:intro2}, we will deduce our main result.

\begin{theorem} \label{th:intro3}
The triple $ (P(\vlambda), R^\ell, q^{\langle |\vlambda|,\rho^\vee \rangle} f_\vlambda(q))$ exhibits the cyclic sieving phenomenon.
\end{theorem}
In the above section $ f_\vlambda(q) $ was only defined for simply-connected $ G$, but in fact the definition can be extended to arbitrary $ G $ using current algebras.

When $ G = SL_n$, the polynomial $ f_\vlambda(q) $ equals the Kostka-Foulkes polynomial.  Thus our theorem reduces to Rhoades' in this case.  In fact, we explicitly compute the power of $ q $ needed to make the Kostka-Foulkes polynomial into the cyclic sieving polynomial, thus answering a question of Rhoades (see Theorem \ref{th:RhoadesRefined}).

The reader may find it curious that we have used two different geometric representation theory tools (affine Grassmannian and quiver varieties) in independent ways.  Indeed, one may interpret our work as proving a non-trivial relationship between these two geometries.  We believe this is an instance of the symplectic duality of Braden-Licata-Proudfoot-Webster \cite{BLPW}.

\subsection{Acknowledgements}
We would like to thank Alexander Braverman, Dennis Gaitsgory, Greg Kuperberg, Anthony Licata, Hiraku Nakajima, and Oded Yacobi for helpful conversations.

\section{Rotation of minuscule Littelmann paths}\label{sec:rotpath}
\subsection{Definition of rotation}

We begin with the following terminology.

We say that $\vgamma = (\gamma_1, \dots, \gamma_m) $ is a \textit{minuscule path} of type $ \vlambda $ if $ \gamma_i - \gamma_{i-1} \in w \lambda_i $ for all $ i $ (where we take $ \gamma_0 = 0 $).  If $\vgamma$ is a minuscule path but is not a minuscule Littelmann path, then we say that $ \vgamma $ is a \textit{non-dominant path}.  If $ \vgamma $ is a non-dominant path, let $\dom(\vgamma)$ be the largest index $i$ such that $\gamma_j$ is dominant for $j<i$.

Let $W$ be the Weyl group of $G$. For a weight $\lambda$ of $G$, let $W_\lambda$ be the stabilizer of $\lambda$ in $W$ and $W\lambda$ be the $W$ orbit of $\lambda$.

Then define the following operation from non-dominant paths to minuscule paths:
$$r(\vgamma)_i = \left\{ \begin{array}{ll} \gamma_i & \mbox{if } i < \dom(\vgamma) \\ \gamma_i +(w\gamma_{\dom(\vgamma)}-\gamma_{\dom(\vgamma)}) & \mbox{otherwise} \end{array} \right. $$ where $w\in W$ is of minimal length such that $w\gamma_{\dom(\vgamma)}$ is dominant. We see that $\dom(r(\vgamma))>\dom(\vgamma)$ and thus there exists a $k$ such that $r^k(\vnu)$ is dominant.  We define $ r^{\mathrm{max}}(\vnu) = r^k(\vnu)$ for this $ k $. 

Let $ \vlambda $ be a minuscule sequence and let $ \vmu \in P(\vlambda) $.  We will now construct $ R(\vmu), $ a minuscule Littelmann path of type $ \vlambda^{(1)}$.

Denote by $\vnu$ the minuscule path $\nu_i=\mu_{i+1}-\mu_1$ of length $m-1$ and type $ (\lambda_2, \dots, \lambda_m)$.  We have $\nu_0=0$ and $\nu_{m-1}=-\mu_1$ and thus it is non-dominant.

Define $R(\vmu)$ by $ R(\vmu)_i = r^{\mathrm{max}}(\vnu)_i $ for $ i < m $ and $ R(\vmu)_m = 0 $.

We claim that $R(\vmu)\in P(\vlambda^{(1)})$. To show this we first need the following lemma:
\begin{lemma}\label{lem:weights} Let $\gamma$, $\beta$ be weights of $G$ with $\beta$ dominant and $\gamma-\beta\in W\lambda$ for some minuscule weight $\lambda$. Then the Weyl group element $w$ of minimal length such that $w\gamma$ is dominant is contained in $W_\beta$.
\end{lemma}

\begin{proof} In general, the length of the minimal $w$ such that $w\gamma$ is dominant is the number positive roots $\alpha$ such that $\braket{\gamma,\alpha^\vee}<0$. If $\gamma$ is dominant, then we are done. Otherwise there exists a simple root $\kappa$ with $\braket{\gamma,\kappa^\vee}<0$. Since $\beta$ is dominant we have $\braket{\beta,\kappa^\vee}\geq 0$ and since $\lambda$ is minuscule $\braket{\gamma-\beta,\kappa^\vee}\in\{-1,0,1\}$. Thus $\braket{\gamma,\kappa^\vee}=-1$ and $\braket{\beta,\kappa^\vee}=0$. This means that $s_\kappa \beta=\beta$ and $s_\kappa\in W_\beta$, so we have $s_\kappa\gamma-\beta\in W\lambda$. Note that for positive roots $\alpha$ we have $\braket{s_\kappa\gamma,\alpha^\vee}=\braket{\gamma,s_\kappa\alpha^\vee}$ and recall that $s_\kappa$ permutes the positive roots that are not $\kappa$ and sends $\kappa$ to $-\kappa$. Thus the minimal $\tilde{w}$ such that $\tilde{w}s_\kappa\gamma$ is dominant satisfies $l(\tilde{w})+1=l(w)$. By induction on $l(w)$, $\tilde{w}\in W_\beta$ and we can take $w=\tilde{w}s_\kappa$.
\end{proof}

\begin{lemma} We have $R(\vmu)\in P(\vlambda^{(1)})$.
\end{lemma}

\begin{proof} We need to check that $R(\vmu)_{i}-R(\vmu)_{i-1}\in W\lambda^{(1)}_i=W\lambda_{i+1}$. First, we check the condition when $i<m$. Given a non-dominant minuscule path $\vgamma$, for $i\neq\dom(\vgamma)$ we have $$r(\vgamma)_i-r(\vgamma)_{i-1}=\gamma_{i}-\gamma_{i-1}$$ and for $i=\dom(\vgamma)$ we can apply the previous lemma $$r(\vgamma)_{i}-r(\vgamma)_{i-1}= w\gamma_{i}-\gamma_{i-1}=w(\gamma_{i}-\gamma_{i-1}).$$ Thus $r$ does not change the Weyl orbit in which the successive differences lie. Since $\vnu$ is a non-dominant minuscule path with $\nu_{i}-\nu_{i-1}=\mu_{i+1}-\mu_i\in W\lambda_{i+1}$, by induction we have $R(\vmu)_{i}-R(\vmu)_{i-1}=r^{\mathrm{max}}(\vnu)_i-r^{\mathrm{max}}(\vnu)_{i-1}\in W\lambda_{i+1}$ for $i<m$.

For $i=m$, we will show that if $\vgamma$ is a non dominant minuscule path such that $\gamma_{\dom(\vgamma)}-\gamma_{m-1}$ is dominant and $\gamma_{m-1}\in W(-\lambda_1)$, then $r(\vgamma)$ satisfies the same properties. By supposition $\braket{\gamma_{\dom(\vgamma)}-\gamma_{m-1},\alpha^\vee}\geq 0$ for all simple roots $\alpha$ and since $\gamma_{m-1}\in W(-\lambda_1)$ we have $\braket{\gamma_{m-1},\alpha^\vee}\in\{-1,0,1\}$. By the bilinearity of the pairing, we then have $\braket{\gamma_{\dom(\vgamma)},\alpha^\vee}\geq -1$ and if $\braket{\gamma_{\dom(\vgamma)},\alpha^\vee}=-1$, then $\braket{\gamma_{m-1},\alpha^\vee}=-1$ as well. We conclude that $$w\gamma_{\dom(\vgamma)}-\gamma_{\dom(\vgamma)}=w\gamma_{m-1}-\gamma_{m-1}$$ where $w$ is the minimal word such that $w\gamma_{\dom(\vgamma)}$ is dominant. Then we have $$r(\vgamma)_{m-1}=\gamma_{m-1}-(w\gamma_{\dom(\vgamma)}-\gamma_{\dom(\vgamma)})=w\gamma_{m-1}\in W(-\lambda_1).$$ Since our starting path $\vnu$ has $\nu_{m-1}=-\lambda_1$ and $$\vnu_{\dom(\vnu)}-(-\lambda_1)=\mu_{\dom(\vnu)+1},$$ by induction $r^{\mathrm{max}}(\vnu)_{m-1}\in W(-\lambda_1)$ and so $-R(\vmu)_{m-1}\in W\lambda_1$ as needed.
\end{proof}

Note that at this point it is not clear that this is actually a rotation, i.e. that $R^m(\vmu)=\vmu$, but this is forced by the connection with the underlying geometry shown in the next section.

\subsection{Connection to promotion on tableaux} \label{se:tableaux}
In the case that $G=\SL_n$, this algorithm has a more classical interpretation. There is a one to one map from $P(\vlambda)$ to a certain set of tableaux and the map is equivariant with respect to promotion on tableaux and rotation of minuscule paths.

\begin{lemma} Let $G=\SL_n$ and $\vlambda=(\omega_{i_1},\omega_{i_2},\dots,\omega_{i_m})$, where $\omega_j$ is the $j$-th fundamental weight of $\SL_n$. There is a bijection between rectangular row strict tableaux with $n$ rows and content $(i_1,i_2,\dots,i_m)$ and the elements of $P(\vlambda)$.
\end{lemma}

\begin{proof} Given a tableaux with $n$ rows and content $\vec{i}$, define $\mu_j$ as the shape of the subtableaux supported on entries $1$ to $j$. $\mu_j-\mu_{j-1}$ records the rows in which the entry $j$ appears, and thus $\mu_j-\mu_{j-1}\in W\omega_{i_j}$. This is a bijection between the two sets.
\end{proof}

Recall that the operation of promotion on row strict tableaux works as follows: delete all boxes labelled $1$, then perform jeu-de-taquin. In the $j$-th step, slide the boxes labelled $j+1$ to the left to fill any gaps and then up to fill gaps. Relabel these boxes $j$. At the end, the last column of the tableaux has $i_1$ empty boxes, which get filled with the number $m$. The relationship between promotion on tableaux and rotation on $P(\vlambda)$ is that the choice of reflection at each step in rotation exactly implements the the upwards slides of jeu-de-taquin.

\begin{theorem}
Under this correspondence, rotation on the elements of $P(\vlambda)$ corresponds to promotion of the associated tableaux.
\end{theorem}

\begin{proof}
Let $\vmu\in P(\vlambda)$, $T$ be its corresponding tableaux and $T^p$ its promotion. We show by induction on the tableaux entries that $R(\vmu)$ is the sequence corresponding to $T^p$. At $j=1$, promotion on the tableaux moves the $i_2$ boxes labelled $2$ into the first column of and relabels them $1$. Then $R(\vmu)_1-R(\vmu)_0$ is the dominant weight $\omega_{i_2}$, so the correspondence agrees up to the entry $1$. Suppose that it agrees up to the entry $j$, then the subtableaux with entries $1$ to $j$ corresponds to the subpath $(R(\vmu)_0,\dots,R(\vmu)_j)$. The boxes currently labelled $j+2$ are in rows specified by $\mu_{j+2}-\mu_{j+1}$. After relabeling with $j+1$ and shifting left, the 'tableaux' on entries $1$ to $j+1$ has corresponding path $$(R(\vmu)_0,\dots,R(\vmu)_j,R(\vmu)_j+(\mu_{j+2}-\mu_{j+1})).$$ The operation of shifting the blocks labelled $j+1$ up to fill any empty spaces corresponds to taking the dominant weight in the Weyl orbit of $R(\vmu)_j+(\mu_{j+2}-\mu_{j+1})$ which is $R(\vmu)_{j+1}$ by construction.
\end{proof}

\subsection{Abstract definition using crystals}
The above operation of rotation is natural from the theory of crystals.  Let $B(\lambda)$ be the highest weight crystal of highest weight $\lambda$ and let $B(\vlambda)=B(\lambda_1)\otimes\cdots\otimes B(\lambda_n)$.  Here, for the moment, $\lambda_j $ are arbitrary dominant weights (not necessarily minuscule).

We will need a few basic definitions and results about crystals.  Let $ B $ be a crystal, which is assumed to be a direct sum of the crystals $ B(\lambda) $.  An element $ b \in B $ is called \textit{highest (resp. lowest) weight} if $ e_i b = 0 $ for all $ i \in I $ (resp. $f_i b = 0 $ for all $ i \in I $).

\begin{lemma} \label{th:highlow}
Let $ B, C $ be crystals.  If $ b \otimes c $ is highest weight in $ B \otimes C $, then $ b $ is highest weight in $ B $.  If $ b \otimes c $ is lowest weight in $ B \otimes C $, then $ c $ is lowest weight in $ C $.
\end{lemma}

Recall also (see \cite{crystalcommutor}) that we have the Schutzenberger involution $ \xi : B \rightarrow B $.  If $ b $ is lowest (highest) weight, then $ \xi(b) $ will be highest (lowest) weight and lie in the same connected component as $ b $.

Now we consider $ B(\vlambda)^G $ the set of elements in $ B(\vlambda) $ which are both highest and lowest weight (equivalently, the set which are highest weight of weight 0).  We will define $ R : B(\vlambda)^G \rightarrow B(\vlambda^{(1)})^G $ as follows.

Let $ b_1 \otimes b_2 \otimes \cdots \otimes b_m \in B(\vlambda)^G $.  Then by Lemma \ref{th:highlow}, $ b_1 $ is highest weight and $ b_2 \otimes \cdots \otimes b_m $ is lowest weight.  We define
$$
R(b_1 \otimes b_2 \otimes \cdots \otimes b_m) = \xi(b_2 \otimes \cdots \otimes b_m) \otimes \xi(b_1)
$$
It is easy to check that this element actually lives in $ B(\vlambda^{(1)})^G $.

Note that $$ R(b_1 \otimes b_2 \otimes \cdots \otimes b_m) = \sigma_{B(\lambda_1), B(\lambda_2) \otimes \cdots \otimes B(\lambda_m))}(b_1 \otimes (b_2 \otimes \cdots \otimes b_m)), $$
where $ \sigma $ denotes the crystal commutor as defined in \cite{crystalcommutor}.

We also could have defined $ R $ as the composition of the natural bijections
\begin{align*}
B(\vlambda)^G &= \Hom(B(0), B(\lambda_1) \otimes B(\lambda_2) \otimes \cdots \otimes B(\lambda_m)))  \\
&\rightarrow \Hom(B(\lambda_1^*), B(\lambda_2) \otimes \cdots \otimes B(\lambda_m)) \\
&\rightarrow \Hom(B(0), B(\lambda_2) \otimes \cdots \otimes B(\lambda_m) \otimes B(\lambda_1)) = B(\vlambda^{(1)})^G
\end{align*}

Now, let us specialize to the case that all $ \lambda_j $ are minuscule.  Then taking the weight defines a bijection $ B(\lambda_j) \cong W \lambda_j $.  Let us write the inverse of this bijection by $ \nu \mapsto b_\nu $.

It is easy to see the following result, which in particular gives a simple proof of the relationship between minuscule Littelmann paths and tensor product multiplicities.
\begin{proposition}
The map $ \mu \mapsto b_{\mu_1} \otimes b_{\mu_2 - \mu_1} \otimes \cdots \otimes b_{\mu_m - \mu_{m-1}} $ defines a bijection $ P(\vlambda) \rightarrow B(\vlambda)^G $.  Moreover, this bijection exchanges rotation on minuscule Littelmann paths with rotation on $B(\vlambda)^G $.
\end{proposition}

\section{Rotation of the Satake basis}
\subsection{Affine Grassmannian}

Let $ \vG $ denote the Langlands dual group of $ G $. We let $\cO=\C[[t]]$ and $\cK=\C((t))$, and we define the \emph{affine Grassmannian} for $\vG$ by $\Gr=\vG(\cK)/\vG(\cO)$. Since weights of $G$ correspond to coweights of $\vG$, we can think of a weight $\lambda$ of $G$ as an element $t^\lambda$ of $\Gr$. For each dominant weight $ \lambda $, we write $\Gr(\lambda)=\vG(\cO)t^\lambda$; these are the orbits of $ \vG(\cO)$ on $\Gr $.

By general principles, the orbits of $ \vG(\cK) $ on $\Gr \times \Gr $ are in bijection with orbits of $ \vG(\cO) $ on $\Gr $.  For $ p,q \in Gr $, we write $ d(p,q) = \lambda $ if $ (p,q) $ lies in the same $\vG(\cK) $ orbit as $ (t^0, t^\lambda)$.  We think of $ d $ as a distance function on $ \Gr $.

We can now define the Satake fibre $F(\vlambda)$ as the variety of based $ m$-gons in $ \Gr $ with edges lengths given by $ \vlambda $,
$$F(\vlambda)=\{(L_0,\dots,L_m)\in\Gr^{m+1}|L_0=L_m=t^0,d(L_{i-1},L_i)=\lambda_i\}.$$

From the geometric Satake correspondence, we have the following result.
\begin{theorem} \label{th:mvbasis}
Let $\vlambda$ be a sequence of dominant minuscule weights. We have a natural isomorphism $$V(\vlambda)^G \cong H_\top(F(\vlambda),\C).$$
\end{theorem}
Moreover, from the work of Haines \cite{H}, the variety $F(\vlambda)$ is pure dimensional, so each irreducible component $Z \subseteq F(\vlambda)$ yields a vector $[Z] \in V(\vlambda)^G$. These vectors form a basis, the Satake basis.

In \cite{us}, we give an explicit description of the irreducible components of $F(\vlambda)$, given by measuring distances from $ t^0 $ to each of the $ L_i$. For $\vmu\in P(\vlambda)$, we define varieties $$Z_\vmu=\{(L_0,\dots,L_m)\in F(\vlambda)|L_i\in\Gr(\mu_i)\}.$$ We show that $F(\vlambda)=\bigsqcup_{\vmu\in P(\vlambda)} Z_\vmu$ and more over we have the following result:

\begin{theorem} Let $\vlambda$ be a dominant minuscule sequence of weights. For each $\vmu\in P(\vlambda)$, $Z_\vmu$ is a dense subset of one irreducible component of $F(\vlambda)$. The induced correspondence is a bijection between $P(\vlambda)$ and the irreducible components of $F(\vlambda)$.
\label{th:haines} \end{theorem}

In \cite{us}, following ideas from Haines \cite{H}, we also gave a more explicit construction of $Z_\vmu$. We begin with the following lemma:
\begin{lemma}\label{lem:element} Let $\mu, \nu$ and $\lambda$ be dominant weights, suppose that $\lambda$ is minuscule and that there exists a $w\in W$ with $\mu+w\lambda=\nu$. The variety $$T(\mu,\lambda,\nu)=\{p\in\Gr(\lambda)|d(t^{-\mu},p)=\nu\}$$ is isomorphic to $M_+(\mu)t^{w\lambda}$ where $M_+(\mu)=Stab_G(t^{-\mu})$ is the standard parabolic for $\mu$ in $G$. Moreover, suppose that $(a,b,c)\in\Gr^3$ with $d(c,b)=\lambda$, $d(b,a)=\mu$ and $d(c,a)=\nu$. Then there exists an element of $g\in\vG(\cK)$ such that $(ga,gb,gc)=(t^0,t^{-\mu},t^{-\nu})$.
\end{lemma}

\begin{proof} The first part is contained in \cite{us} and the second part is a variation on the proof of the first.
\end{proof}

If $\vmu\in P(\vlambda)$, by definition there exists $w_i\in W$ such that $\mu_{i-1}+w_i\lambda_i=\mu_i$. We can then rewrite the definition for $Z_\vmu$ in another way: \begin{equation} \label{eq:twistprod} Z_\vmu=\{(L_0,\dots,L_m)|L_0=t^0,L_1\in\Gr(\lambda_1),L_i\in f_i(L_{i-1})T(\mu_{i-1},\lambda_i,\mu_i)\}.\end{equation} Here $f_i(L_{i-1})$ is an element of $\vG(\cK)$ that brings $(t^0,t^{-\mu_{i-1}})$ to $(L_{i-1},t^0)$. Such an element exists because $L_{i-1}\in\Gr(\mu_{i-1})$, but we must check that the product is independent of choice.

We see that $p\in f_i(L_{i-1})T(\mu_{i-1},\lambda_i,\mu_i)$ satisfies both $d(L_{i-1},p)=\lambda$ and $d(t^0,p)=\mu_i$. In fact, multiplying by $f_i(L_{i-1})$ is an isomorphism between $T(\mu_{i-1},\lambda_i,\mu_i)$ and the set $\{p\in\Gr|d(L_{i-1},p)=\lambda, d(t^0,p)=\mu_i\}$ whose definition is independent of the choice of $f_i(L_{i-1})$. On the other hand, this also shows that if $L_{i-1}\in\Gr_{\mu_{i-1}}$ then $L_i\in f_i(L_{i-1})T(\mu_{i-1},\lambda_i,\mu_i)$ satisfies both $d(L_{i-1},L_i)=\lambda$, $d(t^0,L_i)=\mu_i$. Thus by induction on $i$, the two definitions of $Z_\vmu$ are the same.

\subsection{Rotation of the Satake basis}

As seen in section~\ref{sec:rotpath}, there is a natural map between $P({\vlambda})$ and $P(\vlambda^{(1)})$ given by rotation. By Theorem~\ref{th:haines} these sets give rise to the components of $F(\vlambda)$ and $F(\vlambda^{(1)})$, and hence we have a natural correspondence between the components of these two Satake fibres. On the other hand, we now see that, up to sign, the action of rotation of tensor factors sends the Satake basis of $V(\vlambda)^G$ to the Satake basis of $V(\vlambda^{(1)})^G$.

Let $ Z $ be an irreducible component of $ F(\vlambda) $.  Since $ G(\cO) $ is connected and acts on $ F(\vlambda) $, we see that $ Z $ is $ G(\cO) $ invariant.

Define
$$
\tilde{Z} = \{([g_1^{-1} g_2], \dots, [g_1^{-1} g_n], t^0) : ([g_1], \dots, [g_n]) \in  Z \}  \subset F(\vlambda^{(1)})
$$
(since $ Z $ is $ G(\cO) $-invariant, $ \tilde{Z} $ is well-defined).

It is easy to see that this gives us a bijection
\eq{e:bij}{\Irr(F(\vlambda)) \cong    \Irr(F(\vlambda^{(1)}))}
and compatible isomorphisms of vector spaces
$$ R^{geom} : H_\top(F(\vlambda)) \cong
   H_\top(F(\vlambda^{(1)})).$$
which we call geometric rotation of components.

To compare geometric rotation and rotation of tensor factors, we recall the following result which is a consequence Theorem 4.5 of \cite{us} (see also the commutative diagram (8) in \cite{us}).

\begin{theorem} \label{th:rotation}
Let $ \vlambda $ be a minuscule sequence.  The diagram
\begin{equation}
\begin{tikzpicture}[description/.style={fill=white,inner sep=2pt},baseline]
\matrix (m) [matrix of math nodes, row sep=3em,
column sep=2.5em, text height=1.5ex, text depth=0.25ex]
{ H_\top(F(\vlambda)) & V(\vlambda)^G \\
H_\top(F(\vlambda^{(1)})) & V(\vlambda^{(1)})^G \\ };
\path[->] (m-1-1) edge (m-1-2) edge (m-2-1) (m-2-1) edge (m-2-2) (m-1-2) edge (m-2-2) ;
\end{tikzpicture}
\label{e:diagram} \end{equation}
commutes up to a factor of $ (-1)^{\langle \lambda_1, 2\rho^\vee \rangle}$ where the horizontal edges come from Theorem \ref{th:mvbasis}, the left vertical edge is given by geometric rotation $R^{geom} $ and the right vertical edge is given by rotation of tensor factors.
\end{theorem}

We are now ready to state and prove a precise version of Theorem \ref{th:intro1}.
\begin{theorem}\label{th:rotcomp}
Let $ \vlambda $ be a dominant minuscule sequence and let $ \vmu \in P(\vlambda) $.  Then $ R([\overline{Z_\vmu}]) = (-1)^{\langle \lambda_1, 2\rho^\vee \rangle}[\overline{Z_{R(\vmu)}}]\in H_\top(F(\vlambda^{(1)}))$.
\end{theorem}

Utilizing (\ref{e:diagram}), this theorem reduces to showing that $$R^{geom}([\overline{Z_\vmu}]) = [\overline{Z_{R(\vmu)}}].$$ For this, it is sufficient to see that the subset of $Z_\vmu$ given by $$Y_\vmu=\{(L_0,\dots,L_m)\in Z_\vmu| d(L_1,L_i)=R(\vmu)_{i-1}\}$$ is dense in $Z_\vmu$. To do this we will express $Y_\vmu$ in form of (\ref{eq:twistprod}) using the following lemma:

\begin{lemma} Let $\mu,\nu,\mu',\nu'$ be dominant weights and $\lambda$ a dominant minuscule weight such that there exists $a\in W$ with $\mu'+a\lambda=\nu'$ and $w\in W_{\mu}$ with $w(\nu'-\mu')+\mu=\nu$. The variety $T(\mu,\lambda,\nu)\cap T(\mu',\lambda,\nu')$ is dense in $T(\mu',\lambda,\nu')$.
\end{lemma}

\begin{proof} If $p\in T(\mu,\lambda,\nu)\cap T(\mu',\lambda,\nu')$ then $d(t^0,p)=\lambda$, so $p\in \Gr_{\lambda}=G^\vee/M(\lambda)$. Consider the stabilizers of $t^{-\mu}$ and $t^{-\mu'}$, the standard parabolics $M_+(\mu)$ and $M_+(\mu')$ respectively. We also need $d(t^{-\mu'},p)=\nu'$, so since $\mu'+a\lambda=\nu'$ we have $p\in M_+(\mu')t^{a\lambda}$ and we note that $a$ is the longest element (under the opposite Bruhat order) of its $W_{\mu'}\setminus W/W_{\lambda}$ orbit. On the other hand, we have $\mu+wa\lambda=\nu$ since $w(\nu'-\mu')+\mu=\nu$. Thus we must also have $p\in M_+(\mu)t^{wa\lambda}$ because $d(t^{-\mu},p)=\nu$, but $w\in W_\mu$, so we see that $M_+(\mu)t^{wa\lambda}=M_+(-\mu)t^{a\lambda}$. Considering the decompositions into standard Borel orbits: $M_+(\mu')t^{a\lambda}=\bigsqcup_{a'\in W_{\mu'}a}B_{+}t^{a'\lambda}$ and $M_+(\mu)t^{a\lambda}=\bigsqcup_{a'\in W_{\mu}a}B_{+}t^{a'\lambda}$. The intersection of these two contains $B_{+}t^{a\lambda}$ which is dense in $M_+(\mu')t^{a\lambda}=T(\mu',\lambda,\nu')$.
\end{proof}

We can now give a proof of Theorem~\ref{th:rotcomp}:
\begin{proof}[Proof of Theorem~\ref{th:rotcomp}]
The variety $Y_\vmu$ can now be described as a twisted product:
\begin{multline*}
\{(L_0,\dots,L_m)\in\Gr^m|L_0=t^0,L_1\in\Gr(\lambda_1), \\
 L_i\in f_i(L_1,L_{i-1})T(R(\vmu)_{i-2},\lambda_i,R(\vmu)_{i-1})\cap T(\mu_{i-1},\lambda_i,\mu_i)\}.
\end{multline*}

Here $f_i(L_1,L_{i-1})$ is an element of $\vG(\cK)$ taking $(t^0,t^{-R(\mu)_{i-2}},t^{-\mu_{i-1}})$ to $(L_{i-1},L_1,t^0)$ as given by Lemma~\ref{lem:element}. First, we must see that the above expression is well defined, since $ f(L_1, L_{i-1}) $ is not uniquely defined.  This follows with the same argument as given in the remarks following Lemma~\ref{lem:element}. Next we check that this is indeed the variety $Y_\vmu$. In general $L_i$ is restricted by 3 conditions: $d(L_0=t^0,L_i)=\mu_i$, $d(L_1,L_i)=R(\vmu)_{i-1}$ and $d(L_{i-1},L_i)=\lambda_i$. If the point $L_{i-1}$ satisfies its versions of these three conditions, then the set $f_i(L_1,L_{i-1})T(R(\vmu)_{i-2},\lambda_i,R(\vmu)_{i-1})\cap T(\mu_{i-1},\lambda_i,\mu_i)$ is exactly all possible points in $\Gr$ satisfying the conditions on $L_i$.

To see that $Y_\vmu$ is dense in $Z_\vmu$ we use the previous lemma. Since $\vmu\in P(\vlambda)$, there exists $a_i\in W$ with $\mu_{i-1}+a_i\lambda_i=\mu_i$ and by construction there exists $w_i\in W_{R(\vmu)_{i-2}}$ with $w_i(\mu_i-\mu_{i-1})+R(\vmu)_{i-2}=R(\vmu)_{i-1}$. Thus by the previous lemma, the variety $T(R(\vmu)_{i-2},\lambda_i,R(\vmu)_{i-1})\cap T(\mu_{i-1},\lambda_i,\mu_i)$ is dense in $T(\mu_{i-1},\lambda_i,\mu_i)$ for each $i$ and thus the result is proven.
\end{proof}

\section{$q$-analog of tensor product multiplicity}
\subsection{The quiver variety}
In this section, we specialize to the case where $ G $ is simply-laced.  Let $ \vlambda $ be a minuscule sequence.  From $ \vlambda $, we produce a vector $ \bw = (\bw(\vlambda)_i)_{i \in I} $ by defining
$$
\bw(\vlambda)_i = |\{ 1 \le j \le m : \lambda_j = \omega_i \} |
$$
We fix $ \bw $ for the remainder of this section.

We consider the set
$$ \mathcal V(\bw) := \{ \bv : \sum w_i \omega_i - \sum v_i \alpha_i \text{ is dominant }\} $$
Note that if $ \Hom(V(\mu), V(\vlambda))  \ne 0 $, then $ \mu = \mu(\bv) = |\vlambda| - \sum v_i \alpha_i $ for $ \bv \in \mathcal V(\bw) $.

For any $ \bv $, we consider the smooth Nakajima quiver variety $ \cM(\bv,\bw) $ and the affine singular Nakajima quiver variety $ \cM_0(\bv,\bw) $.

Note that $ \cM_0(\bw) $ contains a most singular point $ 0 $ and we let $ L(\bw) = \pi^{-1}(0) \subset \cM(\bw) $.  Note that $ L(\bw) = \sqcup_\bv L(\bv, \bw)$.

Let $ d_\bv = 2\dim L(\bv, \bw) = \dim \cM(\bv, \bw)$.  A formula for $ d_\bv $ is given in \cite{Nak99}.  We will only need a special case of this formula later.  We let $ H_{[k]}(L(\bw)) = \bigoplus_\bv H_{d_\bv - k}(L(\bv, \bw)) $.

From section 15 of \cite{Nak99}, the decomposition theorem applied to the map $ \pi : \cM(\bw) \rightarrow \cM_0(\bw) $ produces
\begin{equation} \label{eq:NakDecomp}
H_{[k]}(L(\bw)) = \oplus_{\bv \in \mathcal V(\bw)} H_{\top}(\cM(\bw)_{x_\bv}) \otimes H^k(IC_0(\cM_0(\bv,\bw))).
\end{equation}
where $ x_\bv $ is a generic point of $ \cM_0(\bv, \bw)^{reg} $, $\cM(\bw)_{x_\bv}$ denotes the fibre of $ \cM(\bw) $ at the point $ x_\bv $ and $ IC_0(\cM_0(\bv,\bw)) $ denotes the $ !$-stalk of the IC sheaf of $ \cM_0(\bv,\bw) $ at the point $ 0 $.

Moreover, in \cite{Nak99}, the following results are proven.

\begin{theorem}
$H_{[k]}(L(\bw)) = 0 $ if $ k $ is odd and thus $ H^k(IC_0(\cM_0(\bv, \bw))) $ also vanishes if $ k $ is odd.

There exists an action of $ G $ on $ H_*(L(\bw)) $ with the following properties.
\begin{enumerate}
\item $ H_*(L(\bw)) $ is isomorphic to the representation $ V(\vlambda)$.
\item  With the $[*]$-grading, the action of $ G $ preserves degrees, so each $ H_{[k]}(L(\bw)) $ is a subrepresentation.
\item The decomposition (\ref{eq:NakDecomp}) is compatible with the $ G $-action, where, as $G$-reprsentations $ H_{\top}(\cM(\bw)_{x_\bv}) \cong V(\mu(\bv)) $ with $ G $ acting trivially on the stalk of the IC sheaf.
\end{enumerate}
\end{theorem}

It follows from this theorem that for $ \bv \in \mathcal V(\vlambda) $, there is an isomorphism $ H^*(IC_0(\cM_0(\bv,\bw))) \cong \Hom(V(\mu(\bv)), V(\vlambda)) $.

For any $ \bv \in \mathcal V(\bw) $, let $ f_\vlambda^{\mu(\bv)}(q)$ be the Kazhdan-Lusztig type polynomial of $ \cM_0(\bv, \bw) $ at $ 0$.  More precisely, we define
$$
f_\vlambda^{\mu(\bv)}(q) := \sum_{k \ge 0} \dim H^{2k}(IC_0(\cM_0(\bv,\bw))) q^k
$$
We call $ f_\vlambda^\mu(q) $ the $ q$-analog of tensor product multiplicity.  By the above observation $ f_\vlambda^\mu(1) = \dim \Hom(V(\mu), V(\vlambda))$.

The remainder of this section is devoted to proving the following result.

\begin{theorem} \label{th:fullhomology}
There exists an isomorphism of $ G $-representations between $ H_* (L(\bw)) $ and $ V(\vlambda) $, such that $ R^\ell $ acts by $\zeta^k $ on $ H_{[2k]} (L(\bw)) $.
\end{theorem}

As a corollary of this theorem and the above use of the decomposition theorem, we obtain the following result.
\begin{corollary} \label{th:rotationqanalog}
Let $ \bv \in \mathcal V(\bw) $ and let $ \mu = \mu(\bv) $.  There exists an isomorphism of vector spaces
$$ H^*(IC_0(\cM_0(\bv,\bw))) \rightarrow \Hom_G(V(\mu), V(\vlambda)), $$
such that $ R^\ell $ acts by $\zeta^k $ on $ H^{2k}(IC_0(\cM_0(\bv,\bw))) $.

In particular, we see that $ f_\vlambda^\mu(\zeta^d) $ is the trace of $ R^{\ell d} $ on $\Hom_G(V(\mu), V(\vlambda))$.
\end{corollary}
Theorem \ref{th:intro2} is a special case of this result, where we define $ \cM_0(\vlambda) = \cM_0(\bv,\bw) $ where $ \bv $ is defined by $ \sum v_i \alpha_i = |\vlambda|$.

\subsection{Current algebra actions}
The vector space $ H_*(L(\bw)) $ carries an action of the current Lie algebra $ \fg[t] : = \fg \otimes_\C \C[t] $ which extends the action of $ \fg $.  This follows from specializing at the construction of Varagnolo of an action of the Yangian on the equivariant homology of quiver varieties.

More precisely, we have the following result due to Kodera-Naoi \cite{KN}.
\begin{theorem} \label{th:WeylQuiver}
There is an action of $ \fg[t] $ on $H_*(L(\bw))$.  The homological $[*]$-grading is compatible with the grading on $ \fg[t] $.  Moreover, we have an isomorphism of $ \fg[t] $-modules between $ H_*(L(\bw)) $ and the Weyl module $ W(|\vlambda|) $.
\end{theorem}

We will also need the following construction of Fourier-Littelmann \cite{FL} which realizes the Weyl module as a fusion product of evaluation representations.  Let $ c_1, \dots, c_m $ be distinct complex numbers.  Then we have evaluation representations $ V(\lambda_k)_{c_k} $ of $ \fg[t] $.  We may form their tensor product
$$ W := V(\lambda_1)_{c_1} \otimes \cdots \otimes V(\lambda_m)_{c_m} $$
We choose a cyclic vector $w $ for $ W $ by taking a tensor product of $ G $-highest weight vectors in each $ V(\lambda_i) $.  Then we define a filtration $W_0 \subset W_1 \subset \cdots $ on $ W $ by defining $ W_k = (U\fg[t])_k w $, where $ (U\fg[t])_k $ denotes the filtration on $ U\fg[t] $ coming from the grading on $ \fg[t]$.
Finally, we take the associated graded by this filtration.  This results in a graded $ \fg[t] $-module which is called the fusion product and denoted $ V(\lambda_1) * \cdots * V(\lambda_m):= gr(W) $.

The following result is due to Fourier-Littelmann.
\begin{theorem} \label{th:WeylFusion}
For any choice of distinct complex numbers $ c_1, \dots, c_m $, $$ V(\lambda_1) * \cdots * V(\lambda_m) \cong W(|\vlambda|) $$ as graded $ \fg[t] $-modules.
\end{theorem}

We are now in a position to prove Theorem \ref{th:fullhomology} and thus our second main theorem.
\begin{proof}[Proof of Theorem \ref{th:fullhomology}]
Let $ \eta $ denote a primitive $ m$th root of unity such that  $ \eta^\ell =\zeta  $.

Let $ c_i = \eta^i $ for $ i = 1, \dots m $.  Then we consider the tensor product of evaluation $ \fg[t] $-modules
$$
W =  V(\lambda_1)_{c_1} \otimes \cdots \otimes V(\lambda_m)_{c_m}
$$
as above.

We have a linear $ G $-equivariant map $ R^\ell : W \rightarrow W $ given by rotation of tensor factors $ \ell $ times, so
$$
R^\ell (v_1 \otimes \cdots \otimes v_m) = v_{\ell + 1} \otimes v_{\ell + 2} \otimes \cdots \otimes v_{\ell}.
$$
Note that if $ w $ denotes the cyclic vector in $ W $, then $ R^\ell w  = w $.

The map $ R^\ell $ is not $ \fg[t]$-equivariant.  More precisely, if we take $ X t^k \in \fg[t] $, then
\begin{align*}
R^\ell( Xt^k (v_1 \otimes \cdots \otimes v_m)) &= R^\ell(\sum_j \eta^{kj} v_1 \otimes \cdots \otimes X v_j \otimes \cdots \otimes v_m) \\
&= \sum_j \eta^{kj} v_{\ell + 1} \otimes \cdots \otimes X v_j \otimes \cdots \otimes v_\ell \\
&= \eta^{k\ell} Xt^k R^\ell(v_1 \otimes \cdots \otimes v_m)
\end{align*}
This shows that $ R^\ell $ preserves the filtration $ W_k $ and acts by $ \eta^{k\ell} = \zeta^k $ on $ W_k/W_{k-1} $.

Thus on $ gr(W) $, $ R^\ell $ acts by $ \zeta^k $ on the $k$th graded piece and thus $ R^\ell $ acts by $ \zeta^k $ on the $ k$th graded piece of the Weyl module by Theorem \ref{th:WeylFusion} and thus it acts by $\zeta^k $ on $\oplus_{\bv} H_{[2k]} L(\bv, \bw) $ by Theorem \ref{th:WeylQuiver}.
\end{proof}

Another way to think about this proof is to say that we start with $ V(\vlambda) $.  By considering $ R^\ell$, $ V(\vlambda) $ becomes a $ G \times C_r $ representation.

On the other hand, we can consider the $ G([t]) \ltimes \C^\times $ representation $ V(\lambda_1) * \cdots * V(\lambda_m) $, where $ \C^\times $ acts on $ G([t]) $ by loop rotation and on $ V(\lambda_1) * \cdots * V(\lambda_m) $ by means of the grading.  Then above we have proven that
$$
 V(\lambda_1)* \cdots *V(\lambda_m) \cong V(\vlambda)
$$
as $ G \times C_r $ representations, where we use $ G \times C_r \hookrightarrow G([t]) \ltimes \C^\times $ in order to view the LHS as a $ G \times C_r $ representation.

The above proof works if $ G $ is not simply-laced and also if the $ \lambda_k $ are not minuscule.  Thus, we get an analog of Theorem~\ref{th:fullhomology} and Corollary~\ref{th:rotationqanalog}, where the $ q$-analog of tensor product multiplicity is defined using the grading on the fusion product $ V(\lambda_1) * \cdots * V(\lambda_m)$.

\subsection{Completion of proof of Theorem \ref{th:intro3}}
Let $ \vlambda $ denote a sequence of dominant minuscule weights.  By Theorem \ref{th:intro1}, the linearization of the map $ R^\ell : P(\vlambda) \rightarrow P(\vlambda) $ is the linear operator $(-1)^{\langle \lambda_1 + \cdots + \lambda_\ell,2\rho^\vee \rangle}R^\ell \text{ on } V(\vlambda)^G $.
By Theorem \ref{th:intro2} (in particular Corollary \ref{th:rotationqanalog}), the trace of $ R^{d\ell} $ is given by $ f_{\vlambda}(\zeta^d) $.

Thus, we see that the trace of $ (-1)^{d\langle \lambda_1 + \cdots + \lambda_\ell ,2\rho^\vee \rangle} R^{d\ell} $ is given by
$$ (\zeta^{r/2})^{d \langle \lambda_1 + \cdots + \lambda_\ell,2\rho^\vee \rangle} f_{\vlambda}(\zeta^d) = (\zeta^d)^{\langle |\vlambda|, \rho^\vee \rangle}f_{\vlambda}(\zeta^d)$$
where we use that $ r(\lambda_1 + \cdots + \lambda_\ell) = |\vlambda|$.  Thus, we can take $ q^{\langle |\vlambda|, \rho^\vee \rangle }f_{\vlambda}(q) $ as our cyclic sieving polynomial.

\section{Special case of $ SL_n$}
In this section, we take $ G = SL_n $.  We will use skew Howe duality to see that the above $ q$-analog of tensor product multiplicity for $ SL_n $ is actually $q$-analog of weight multiplicity for $ SL_m $.

\subsection{Application of (geometric) skew Howe duality}
Let $ \vlambda $ be a minuscule sequence as before.  For each $ j $, we can write $ \lambda_j = \omega_{i_j} $ for some $ i_j \in \{1, \dots, n-1 \} $.  Let $ N = \sum i_j $.

Now, $ V(\vlambda) $ is actually a weight space in a representation of $ SL_m $.  By elementary linear algebra, we have an identification
$$ V(\vlambda) = \Lambda^N(\C^n \otimes \C^m)_{(i_1, \dots, i_m)}. $$  Skew Howe duality shows that this specializes to an identification
\begin{equation} \label{eq:skewHowe}
 \Hom_{SL_n}(V(\mu), V(\vlambda)) = V(\nu)_\gamma
\end{equation}
where $ \nu = \mu^\vee $ denotes the transpose partition to $ \mu $ and $ \gamma = (i_1, \dots, i_m) $.

In the previous section, we saw that $\Hom_{SL_n}(V(\mu), V(\vlambda))$ equals the $!$-stalk at 0 of the IC sheaf of $ \cM_0(\bv, \bw) $ (where $ \bv, \bw $ are constructed from $ \vlambda $ in the previous section).   We will now see that the right hand side of (\ref{eq:skewHowe}) can also be identified with the $!$-stalk of an IC sheaf.

Recall the Kostka-Foulkes polynomial $ K_{\nu, \gamma}(q) $ which may be defined using the $ q$-analog of the Kostant multiplicity formula.  In particular $ K_{\nu, \gamma}(1) = \dim V(\gamma)_\nu $ and we refer to $ K_{\nu, \gamma}(q) $ as the $ q$-analog of this weight multiplicity.

From Lusztig's work \cite{Lusztig:qanalog}, we know that $ V(\nu)_\gamma $ equals the stalk at $t^\gamma $ of the IC sheaf of $ \Gr(\nu) $ (here we work in the affine Grassmannian of $ PGL_m $).  Lusztig \cite{Lusztig:qanalog} conjectured, and Kato \cite{K} proved, that the Kostka-Foulkes polynomial $ K_{\nu, \gamma}(q) $ is related to this stalk by
\begin{equation} \label{eq:KFLusztig}
K_{\nu, \gamma}(q) = q^{\langle-\gamma, \rho  \rangle} \sum_{k \ge 0} \dim H^k(IC_{t^\gamma}(\Gr(\nu))) q^{k/2}
\end{equation}

Let us consider the transverse slice $ \Gr(\nu)_\gamma $ (see \cite{Mirkovic-Vybornov} for the precise definition).  Since $ \Gr(\nu)_\gamma $ is a transverse slice we have that $ i^! IC(\Gr(\nu)) = IC(Gr(\nu)_\gamma)[-2\langle \gamma, \rho \rangle] $.  Combining this with (\ref{eq:KFLusztig}), we see that the Kostka-Foulkes polynomial is the generating function for the $!$-stalk of the IC sheaf of $ \Gr(\nu)_\gamma $ at $ t^\gamma $,
\begin{equation} \label{eq:KFstalk}
K_{\nu, \gamma}(q) = \sum_{k \ge 0} \dim H^{2k}(IC_{t^\gamma}(\Gr(\nu)_\gamma)) q^k.
\end{equation}

The geometric version of (\ref{eq:skewHowe}) is the following result of Mirkovi{\'c}-Vybornov \cite{Mirkovic-Vybornov}.

\begin{theorem}
With the above notation, there is an isomorphism of varieties $\cM_0(\bv, \bw) \cong \Gr(\nu)_\gamma $ which takes $ 0 $ to $ t^\gamma $.
\end{theorem}

Combining this theorem with (\ref{eq:KFstalk}), we see that the $q$-analog of the tensor product multiplicity $\dim \Hom_{SL_n}(V(\mu), V(\vlambda)) $ equals the $ q$-analog of the weight multiplicity $ \dim V(\nu)_\gamma $.  In other words $$ f_\vlambda^\mu(q) = K_{\nu, \gamma}(q). $$

(Actually, one can see this result without appealing to any geometry.  The $q$-analog of weight multiplicity can be defined using the principal nilpotent in $ \mathfrak{sl}_m$ and one can observe that under the skew Howe duality this principal nilpotent becomes the degree operator related to $ \mathfrak{sl}_n[t] $.)

\subsection{Combinatorial implication}

From our main result Theorem \ref{th:intro3}, the cyclic sieving polynomial for $R^\ell $ on $ P(\vlambda) $ is given by
$$ q^{\langle |\vlambda|, \rho^\vee_{SL_n} \rangle }f_{\vlambda}(q) $$

Combining all this together, setting $ \mu = 0 $, and switching to tableaux with the aid of section \ref{se:tableaux} we obtain the following result which is a more precise version of Theorem 1.5 from \cite{R}.

\begin{theorem} \label{th:RhoadesRefined}
  Let $ \gamma = (i_1, \dots, i_m) $ be a sequence of integers with $ 1 \le i_k \le n $ for all $ k $ and with $ i_1 + \cdots +i_m = nb $ for some $ b $.  Assume that $ (i_1, \dots, i_m) $ is invariant under $ \ell$-th cyclic shift.  Let $ T(\gamma) $ be the set of row-strict semistandard Young tableaux of shape $\nu =  n^b $ and content $ \gamma $.  Let $ p $ denote the promotion operator on tableaux.  Then
$$ (T(\gamma), p^\ell, q^{\frac{1}{2}(n^2b - \sum_{j=1}^m i_j^2)}  K_{\nu, \gamma}(q))$$
satisfies the cyclic sieving phenomenon.
\end{theorem}


\begin{thebibliography}{[AAAA]}
\bibitem[BLPW]{BLPW}
T. Braden, A. Licata, N. Proudfoot, B. Webster, Algebraic and geometric category O, in preparation.

\bibitem[FKK]{us}
B. Fontaine, J. Kamnitzer, and G. Kuperberg, Buildings, spiders and geometric Satake; \mbox{arXiv:1103.3519v2}.

\bibitem[FL]{FL}
G. Fourier, P. Littelmann, Weyl modules, Demazure modules, KR-modules, crystals, fusion products and limit constructions,  \emph{Adv. Math.} \textbf{211} (2007), no.~2, 566--593.

\bibitem[G]{Ginzburg:loop}
V. Ginzburg, Perverse sheaves on a loop group and {Langlands} duality; \mbox{arXiv:alg-geom/9511007}.

\bibitem[H]{H}
T. J. Haines, Equidimensionality of convolution morphisms and applications to saturation problems, \textit{Adv. Math.} \textbf{207} (2006), no. 1, 297--327; \mbox{arXiv:math/0501504}.

\bibitem[HK]{crystalcommutor}
A. Henriques, J. Kamnitzer, Crystals and coboundary categories, \emph{Duke Math. J.} \textbf{132} (2006), 191--216.

\bibitem[K]{K}
S. Kato,  Spherical functions and a q-analogue of Kostant's weight multiplicity formula,
\textit{Invent. Math.} \textbf{66} (1982), 461--468.

\bibitem[KN]{KN}
R. Kodera, K. Naoi, Loewy series of Weyl modules and the Poincar\'e polynomials of quiver varieties; \mbox{arXiv:1103.4207}.

\bibitem[L]{Lusztig:qanalog}
G. Lusztig, Singularities, character formulas, and a $q$-analog of weight multiplicities, in: \emph{Analysis and topology on singular spaces, {II}, {III} ({Luminy}, 1981)}, pp.~208--229, Ast\'erisque, vol. 101, Soc. Math. France, 1983.
	
\bibitem[MVi]{MV:geometric}
I.~Mirkovi{\'c}, K.~Vilonen, Geometric {Langlands} duality and representations of algebraic groups over commutative rings, \emph{Ann. of Math. (2)} \textbf{166} (2007), no.~1, 95--143; \mbox{arXiv:math/0401222}.

\bibitem[MVy]{Mirkovic-Vybornov}
I. Mirkovi{\'c}, M. Vybornov, Quiver varieties and Beilinson-Drinfeld Grassmannians of type A; \mbox{arXiv:0712.4160v2}.

\bibitem[N]{Nak99}
H. Nakajima, Quiver varieties and finite dimensional representations of quantum affine algebras, \emph{J. Amer. Math. Soc.} \textbf{14} (2001), 145-238.

\bibitem[PPR]{PPR}
T. K. Petersen, P.~Pylyavskyy, and B.~Rhoades, Promotion and cyclic sieving via webs, \emph{J. Algebraic Combin.} \textbf{30} (2009), no.~1, 19--41.


\bibitem[R]{R}
B.~Rhoades, Cyclic sieving, promotion, and representation theory, \emph{J. Combin. Theory Ser. A} \textbf{117} (2010), no.~1, 38--76.

\end{thebibliography}
\end{document}